\documentclass[11pt]{article}
\usepackage[round]{natbib}
\usepackage{mathrsfs}
\usepackage{latexsym,lineno}
\usepackage{epsfig}
\usepackage{color}
\usepackage{amsmath}\usepackage{fleqn}\usepackage{verbatim}\usepackage{epsf}
\usepackage{amsthm}\usepackage{graphicx, float}\usepackage{graphicx}
\usepackage{amsfonts}\usepackage{amssymb}\usepackage{graphpap}
\usepackage{epic}\usepackage{curves}
\newtheorem{fact}{Theorem}

\newtheorem{theorem}{Theorem}[section]
\newtheorem{word}[fact]{Definition}

\newtheorem{example}{Example}
\newtheorem{corollary}{Corollary}[section]
\newtheorem{lemma}{Lemma}[section]

\begin{document}
\title{\bf Characterizations of the position value for hypergraph communication situations}
\author{Erfang Shan\thanks{{\em Corresponding authors.} Email address: efshan@shu.edu.cn (E. Shan), guangzhang@shu.edu.cn (G. Zhang)}, \,
Guang Zhang\\
{\small School of Management, Shanghai University, Shanghai
200444, P.R. China}
}
\date{}
\maketitle \baselineskip 17pt
\begin{abstract}

We characterize the position value for arbitrary hypergraph communication situations.
The position value is first presented by the Shapley value of the
  uniform hyperlink game or the  $k$-augmented uniform hyperlink game, which are obtained from the given hypergraph
communication situation. These results generalize the  non-axiomatic characterization of the position value for communication situations in \cite{ko} (Int J Game Theory (2010) 39: 669--675)
to hypergraph communication situations.
 Based on the non-axiomatic characterization, we further provide an axiomatic characterization of the position value for arbitrary hypergraph communication situations by employing component efficiency and a new property, named partial balanced conference contributions. The partial balanced conference contributions is developed from balanced link contributions in \cite{sl} (Int J Game
Theory (2005) 33: 505-514).


 \medskip

\end{abstract}
\section{Introduction}

The study of TU-games with limited cooperation presented by means of a communication graph was initiated by \cite{my1},
and an allocation rule for such games, the so-called {\em Myerson value}, was also introduced simultaneously.
Later on, various studies in this direction were done in the past nearly forty years, such as \cite{Me}, \cite{he}, \cite{blp}, \cite{bkl}, \cite{brs} and \cite{szd}. Among them, the allocation rule, named {\em position value} (\cite{Me}), is also widely studied for (graph) communication situations.
\cite{Bo} provided a characterization of the position value for (graph) communication situations with trees.
An elegant characterization of this rule for arbitrary (graph) communication situations was given by \cite{sl}. While \cite{van} extended the position value to TU-games with hypergraph communication situations, shortly hypergraph communication situations.
 They also gave an axiomatic characterization of the position value for cycle-free hypergraph communication situations.
\cite{al} extended the position value to union stable systems and characterized it for a subclass of such systems.
However, an axiomatic characterization of the position value for arbitrary hypergraph communication situations has not yet
been found and remains an open problem.
Additionally, the approach of non-axiomatic characterization for the position value, which deserves to be mentioned, was investigated in \cite{ca1} and \cite{ko}, respectively. \cite{ca1} gave a characterization of the
position value by the Myerson value of a modification of communication situations,
called the {\em  link agent form} (LAF) on graph situations and the {\em hyperlink agent form} (HAF) on hypergraph situations; While, \cite{ko} provided unified and non-axiomatic characterizations of the position value and the Myerson
value by using the {\em divided link game} and the {\em divided link game with a coalition structure}, respectively. However, we note that the approach due to \cite{ko} for communication situations does not work  directly for hypergraph communication situations.

The main aim of this paper is to  provide both non-axiomatic and axiomatic characterization of the position value for arbitrary hypergraph communication situations.
To complete the non-axiomatic characterization of the position value, we introduce two new  games obtained from the original hypergraph communication situation, called the {\em uniform hyperlink game} and the {\em $k$-augmented uniform hyperlink game}, respectively.
It turns out that the position value for hypergraph communication situations can be represented by the Shapley value of the uniform hyperlink game or the $k$-augmented uniform hyperlink game.
Based on the above non-axiomatic characterizations, an axiomatic characterization for arbitrary hypergraph communication situations is proposed by {\em component efficiency} and a new property, called {\em partial balance conference contributions}. Component efficiency states that for each component of the hypergraph
the total payoff to its players equals the worth of that component.
Partial balanced conference contributions is developed from {\em balanced link contributions} which is used to characterize the position value for (graph) communication situations in \cite{sl}.
The partial balanced conference contributions here deals with the payoff difference a player experiences if another player breaks one of his hyperlinks. The intrinsical difference between the two balanced properties is whether the payoff difference of a player experiences is totally  or partially attributing to another player.

This article is organized as follows.  Basic definitions and notation are given in Section 2.
Section 3 first introduces the uniform hyperlink game and the $k$-augment uniform hyperlink game. By using the two games, we  give two non-axiomatic characterizations of the position value for hypergraph communication situations. Further, we present a key property, called partial balanced conference contributions. Based on the non-axiomatic characterizations, we  provide an axiomatic characterization of the position value for arbitrary hypergraph communication situations by  employing component efficiency and partial balanced conference contributions. Finally, we conclude in Section 4 with some remarks.

\section{Basic definitions and notation}
In this section, we recall some definitions and concepts related to TU-games and allocation rules for hypergraph communication situations.

A {\em cooperative game with transferable utility}, or simply a {\em TU-game}, is a pair $(N,v)$
where $N=\{1, 2,\ldots, n\}$ is a finite set of $n\geq 2$ players and $v: 2^N\rightarrow \mathbf{R}$ is a {\em characteristic
function} defined on the power set of $N$ such that $v(\emptyset)=0$.
For any $S\subseteq N$, $S$ is called a {\em coalition} and the real number $v(S)$ represents its {\em worth}.
A {\em subgame } of $v$ with a nonempty set $T\subseteq N$ is a game $v_T(S)=v(S)$, for all $S\subseteq T$.
We denote by $|S|$  the cardinality of  $S\subseteq N$. A game $(N,v)$ is {\em zero-normalized} if for any $i\in N$, $v(\{i\})=0$. Throughout this paper, we consider only zero-normalized games.

Let $\Sigma(N)$ be the set of all permutations on $N$. For any permutation $\sigma\in \Sigma(N)$, the corresponding marginal vector $m^\sigma(N,v)\in \mathbf{R}^n$ assigns to
every player $i\in N$ a payoff $m^\sigma_i(N, v)=v(\sigma^i\cup\{i\})-v(\sigma^i)$, where $\sigma^i=\{j\in N ~|~ \sigma(j)<\sigma(i)\}$
is the set of players preceding $i$ in the permutation $\sigma$. The best-known single-valued solution,
the {\em Shapley value} (\cite{sha}), assigns to any game $(N,v)$ the average  of all marginal vectors.  Formally, the Shapley value  is defined as
follows.
\begin{eqnarray*}
Sh_i(N,v)=\frac{1}{|\Sigma(N)|}\sum_{\sigma\in \Sigma(N)}m^{\sigma}_i(N, v), \, \, \mbox{ for all }\, \, i\in N.
\end{eqnarray*}
An alternative description of the Shapley value can be provided by employing the {\em Harsanyi dividends}.
First, the {\em unanimity game} $(N,u_T)$ according to $T\subseteq N$ is the game defined by $u_T(S)=1$ if $T\subseteq S$ and $u_T(S)=0$ otherwise (\cite{sha}).
Then each game $(N,v)$ can be written as a unique linear combination of unanimity games, i.e., $v=\sum_{T\in 2^N\setminus\{\emptyset\}}\lambda_T(v)u_T$ where $\lambda_T(v)=\sum_{S\subseteq T: S\neq \emptyset}(-1)^{|T|-|S|}v(S)$ is called Harsanyi dividends (\cite{har}) of the nonempty coalition $T\subseteq N$.
The alternative description of the Shapley value is given as follows.
\begin{eqnarray*}
Sh_i(N,v)=\sum\limits_{T\subseteq N: i\in T}\frac{\lambda_T(v)}{|T|}, \, \text{ for all } i\in N.
\end{eqnarray*}

The communication
possibilities for a TU-game $(N,v)$ can be described by a  (communication) {\em hypergraph} $(N, H)$ where  $H$ is a family
of non-singleton subsets of $N$, i.e., $H\subseteq H^N=\{e\subseteq N \,|\,  |e|>1\}$.
The elements of $N$ are called the {\em nodes} or {\em vertices} of the hypergraph that represent players,
and the elements of $H$ its {\em hyperlinks} or {\em hyperedges} represent {\em conferences}
in which all players in a hyperlink have to be present before communication can take place (\cite{van}).
A hypergraph $(N,H)$ is called  $r$-{\em uniform} if $|e|=r$ for all $e\in H$.
Clearly, a graph $(N, L)$, $L\subseteq L^N=\{e\subseteq N\,|\,  |e|=2\}\subseteq H^N$, is a 2-uniform hypergraph and in this case these hyperlinks are called {\em links}. Therefore,  hypergraphs are a natural generalization of graphs in which ``edges"  may consist of more than 2 nodes.

Let $H_i$ be the set of hyperlinks containing player $i$ in a hypergraph $(N, H)$, i.e.
$H_i=\{e\in H\,|\,  i\in e\}$. The {\em degree} of $i$ is defined as $d(i)=|H_i|$.  A node $i\in N$ is {\em incident}
with a hyperlink $e\in H$, if $i\in e$.
Two nodes $i$ and $j$ of $N$ are {\em adjacent} in the hypergraph $(N,H)$ if there is an hyperlink $e$ in $H$ such that $i,j\in e$. Two nodes $i$ and $j$ are {\em connected} if there exists a sequence $i=i_0, i_{1}, \ldots, i_{k}=j$ of nodes of $(N,H)$ in which $i_{l-1}$  is adjacent to $i_{l}$ for $l=1, 2, \ldots, k$. A {\em connected hypergraph} is a hypergraph in which every pair of nodes are connected.
Given any hypergraph $(N, H)$, a (connected) {\em component} of $(N,H)$  is a maximal set of nodes of $N$ in which every pair of nodes are connected.
Let $N/H$ denote the set of components in $(N,H)$ and $(N/H)_i$ be the component containing $i\in N$. For any $S\subseteq N$, let $(S, H(S))$
be the {\em  subhypergraph} induced by $S$ where $H(S)=\{e\in H\,|\,  e\subseteq S\}$.  A hypergraph $(N, H')$ is called a {\em partial hypergraph} of
$(N, H)$ if $H'\subseteq H$. The notation $S/H(S)$ (or for short $S/H$) and $N/H'$ are defined similarly.

A {\em hypergraph communication situation}, or simply a {\em hypergraph game}, is a triple $(N,v,H)$ where $(N,v)$ is a zero-normalized TU-game
and $H$ is  the set of hyperlinks  in the hypergraph $(N,H)$. In particular,
if $(N,H)$ is a graph, the  triple $(N,v,H)$ is called a  {\em communication situation}, or for short a {\em graph game}.
 Let $HCS^N$ denote the class of all hypergraph communication situations
with fixed player set $N$.

Two values, or allocation rules, were well defined for hypergraph communication situations.
The {\em Myerson value} $\mu$ (\cite{my1, my2}, \cite{van}) is defined by
\begin{eqnarray*}
\mu(N,v, H)=Sh(N, v^H),
\end{eqnarray*}
where $v^H(S)=\sum_{T\in S/H}v(T)$ for any $S\subseteq N$, and the game $(N, v^H)$ is called the {\em  point game} or {\em hypergraph-restricted game}.

An alternative value for hypergraph communication situations, the {\em position value} $\pi$ (\cite{Me}, \cite{van}), is given by
\begin{eqnarray*}
\pi_i(N,v, H)=\sum_{e\in H_i}\frac{1}{|e|}Sh_e(H, v^N), \, \, \mbox{for any}\, \, i\in N,
\end{eqnarray*}
where $v^N(H')=\sum_{T\in N/H'}v(T)$ for any $H'\subseteq H$. The game $(H, v^N)$ is called the {\em conference game} or {\em hyperlink game}.

\section{The characterizations of the position value}
In this section we shall provide two kinds of characterizations of the position value for hypergraph communication situation. The first kind of characterizations is non-axiomatic characterizations
of the position value.
The position value is  presented by the Shapley value of the
  uniform hyperlink game or the  $k$-augmented uniform hyperlink game, which are obtained from the given hypergraph
communication situation.
Another one is an axiomatic characterization, which can be proven by employing component efficiency and a new property, called partial balanced conference contributions.

\subsection{The non-axiomatic characterization of the position value}
In order to show the non-axiomatic characterization, we first introduce
 the {\em uniform hyperlink game} induced by an original hypergraph communication situation. The definition of the uniform hyperlink game follows the spirits of the {\em divided link game} in \cite{ko} and the {\em hyperlink agent form} (HAF) in \cite{ca1}.

\begin{word}\label{de1}
For any $(N,v,H)\in HCS^N$ without isolated player, its {\bf uniform hyperlink game} $(U(H), w)$ is defined as follows: Let $\eta(H)$ denote the least common multiple of the numbers in $\{|e|\,|\,e\in H\}$. Then set
\begin{eqnarray}
&\,&\ U(H)(i,e)=\{(i,e,k)\,|\,k=1,2, \ldots,  \eta(H)\cdot|e|^{-1}\},\\
&\,&\ U(H)(i)=\bigcup_{e\in H_i}U(H)(i,e)\, \text{ and }\,  U(H)(e)=\bigcup_{i\in e}U(H)(i,e), \nonumber\\
&\,&\ U(H)=\bigcup_{i\in N}U(H)(i)=\bigcup_{e\in H}U(H)(e)\\
&\,&\ w(S)=v^N(H[S])=\sum\limits_{R\in N/H[S]}v(R), \, \text{ for all } S\subseteq U(H),
\end{eqnarray}
where $H[S]=\{e\in H \,|\, U(H)(e)\subseteq S\}$.
\end{word}
By the definitions of $(U(H),w)$, it is clear that $U(H)$ is obtain from $(N, H)$ by expanding each $i\in N$ to $\eta(H)\cdot|e|^{-1}$
nodes according to each hyperlink $e\in H_i$.
Every set $U(H)(i,e)$ consists of precisely $\eta(H)\cdot|e|^{-1}$ players of $U(H)$
which are obtained by expanding the player $i\in e$ in hyperlink $e\in H_i$.
$U(H)(i)$ is the set of players obtained by expanding $i\in N$ in all hyperlinks $H_i$ and
$U(H)(e)$ is the set of players  obtained by expanding the members of $e\in H$. Therefore, it is easy to check that
$|U(H)|=\eta(H)\cdot|H|$ and $|U(H)(e)|=\eta(H)$ for any $e\in H$.

{\bf Remark}. The definition of $U(H)$ is similar to the player set of HAF in \cite{ca1}, but there are no hyperlinks or links in $U(H)$ as defined in HAF, and the characteristic functions are different from each other as well. The characteristic function $w$ here follows the idea from the {\em divided link game}, due to \cite{ko}. The uniform hyperlink game generalizes the divided link game for graph games to hypergraph games.

For a hypergraph game $(N, v, H)$, we present the following characterization of the position value in terms of the Shapley value of the uniform hyperlink game $(U(H),w)$.
\begin{theorem}\label{thm1}
For any hypergraph game $(N,v, H)$ and any $i\in N$,
$$\pi_i(N,v,H)=\sum_{l\in U(H)(i)}Sh_l(U(H),w).$$
\end{theorem}
\begin{proof}
Let $g$ be a mapping from $\Sigma(U(H))$ to $\Sigma(H)$: For any two hyperlinks $e_1, e_2\in H$ and any permutation $\sigma\in \Sigma(U(H))$, $g(\sigma)(e_1)<g(\sigma)(e_2)$
if and only if $\max\{\sigma(l)\,|\,  l\in U(H)(e_1)\}<\max\{\sigma(l)\,|\,  l\in U(H)(e_2)\}$.
Therefore, for any $l\in U(H)(e)\subseteq U(H)$,  if $\sigma(l)$=$\max\{\sigma(k)\,|\,  k\in U(H)(e)\}$, then
$$m^{\sigma}_{l}(U(H),w)= m^{g(\sigma)}_{e}(H, v^N),$$
otherwise, we have $m^{\sigma}_{l}(U(H),w)=0$.
Hence,
$$\sum_{l\in U(H)(e)}m^{\sigma}_{l}(U(H),w)=m^{g(\sigma)}_{e}(H, v^N).$$

It is clear that $|\Sigma(U(H))|=\big(\eta(H)\cdot|H|\big)!$ and $|\Sigma(H)|=|H|!$. For
each $\delta\in \Sigma(H)$, $\Sigma(U(H))$ has exactly $q=|\Sigma(U(H))|/|\Sigma(H)|$ permutations $\sigma_1, \sigma_2, \ldots, \sigma_{q}$ such that $g(\sigma_t')=\delta$ for $t=1,2, \ldots, q$.
So, we have
\begin{eqnarray}\label{formu1}
&\, & ~~~~\frac{1}{|\Sigma(U(H))|}\sum_{\sigma\in \Sigma(U(H))}\Big(\sum_{l\in U(H)(e)}m^{\sigma}_{l}(U(H),w)\Big)\nonumber\\
&\, &\ =\sum_{l\in U(H)(e)}\Big(\frac{1}{|\Sigma(U(H))|}\sum_{\sigma\in \Sigma(U(H))}m^{\sigma}_{l}(U(H),w)\Big)\nonumber\\
&\, &\ =\frac{1}{|\Sigma(H)|}\sum_{\delta\in \Sigma(H)}m^{\delta}_{e}\big(H, v^N\big).
\end{eqnarray}
Therefore, by the definition of the Shapley value,  we obtain
\begin{eqnarray}\label{formu5}
\sum_{l\in U(H)(e)}Sh_{l}(U(H), w)=Sh_{e}\big(H,v^N\big).
\end{eqnarray}
Note that $|e|> 1$ for each $e\in H$, so there exist at least two players $l,l'\in U(H)(e)$ such that
$w(S\cup\{l\})=w(S)=w(S\cup\{l'\})$ for any $S\subseteq U(H)\setminus\{l,l'\}$. So $l$ and $l'$ are symmetric in $(U(H),w)$.
From the symmetry of the Shapley value, it follows that $Sh_{l}(U(H),w)=Sh_{l'}(U(H),w)$ for any two
links $l, l'\in U(H)(e)$. By Eq. (\ref{formu5}), for any $l\in U(H)(e)$, we have
$$Sh_{l}(U(H), w)=\frac{1}{\eta(H)}Sh_{e}\big(H,v^N\big).$$
Note that, for each $e\in H$, $|U(H)(e)|=\eta(H)$. Therefore, we have
$$\sum_{l\in U(H)(i,e)}Sh_{l}(U(H), w)=\frac{|U(H)(i,e)|}{\eta(H)}Sh_{e}\big(H,v^N\big)=\frac{1}{|e|}Sh_{e}\big(H,v^N\big).$$
The second equation holds by following Eq. (1). Consequently,
$$\pi_i(N,v,H)=\sum_{e\in H_i}\frac{1}{|e|}Sh_{e}\big(H,v^N\big)= \sum_{e\in H_i}\sum_{l\in U(H)(i,e)}Sh_{l}(U(H), w) =\sum_{l\in U(H)(i)}Sh_l(U(H),w).$$
This completes the proof of Theorem \ref{thm1}. \end{proof}

The following example illustrates  the construction of the uniform hyperlink game and  Theorem \ref{thm1}.
\begin{figure}[h]
\centering
\includegraphics[height=5cm,width=10cm]{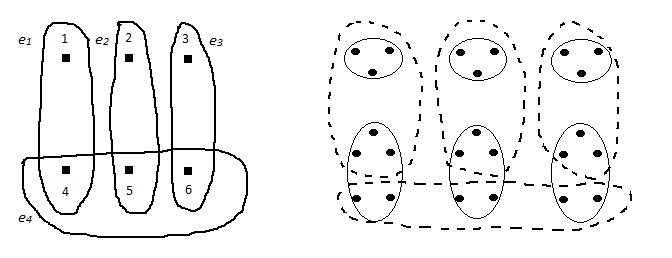}
\caption{The hypergraph $(N,H)$ in Example 1 and its corresponding set $U(H)$.}
\end{figure}

\begin{example}\label{ex1}
Consider the communication game $(N,v,H)$, where $N=\{1,\ldots, 6\}$, $v$ is the unanimity game $u_{\{1,2,3\}}$ and $H=\{\{1,4\},\{2,5\},\{3,6\},\{4,5,6\}\}$. Let $e_1=\{1,4\}, e_2=\{2,5\}, e_3=\{3,6\}, e_4=\{4,5,6\}$. By the definition of $U(H)$, we have the following set (also see Fig. 1)
\begin{eqnarray*}
U(H)\quad =  &\ \{ (1,e_1,1), (1,e_1,2), (1,e_1,3); (4,e_1,1), (4,e_1,2), (4,e_1,3); \\
&\ (2,e_2,1), (2,e_2,2), (2,e_2,3); (5,e_2,1), (5,e_2,2), (5,e_2,3);   \\
&\ (3,e_3,1), (3,e_3,2), (3,e_3,3); (6,e_3,1), (6,e_3,2), (6,e_3,3); \\
&\ \,(4,e_4,1), (4,e_4,2); (5,e_4,1), (5,e_4,2); (6,e_4,1), (6,e_4,2)\, \}.
\end{eqnarray*}
Note that Figure 1 shows the hypergraph $(N,H)$ and its corresponding set $U(H)$.
In the set of $U(H)$, the subsets in which the nodes are encircled by solid lines are developed by the elements of $N$
and indicated by $U(H)(i)$ for all $i\in N$, while the groups where the nodes are encircled by dotted lines
are derived by the players in each hyperlink of $H$ and represented by $U(H)(e)$ for all $e\in H$.

By the definition of $v$, we can calculate the uniform hyperlink game $w$ as follows.
\begin{eqnarray*}
&\,&\ w(S)=\left\{ \begin{array}{cl}
               1 & S=U(H), \\
               0 & \text{else}.
             \end{array}
             \right.
\end{eqnarray*}
Hence, $Sh_l(U(H),w)=\frac{1}{24}$ for all $l\in U(H)$ and
\begin{eqnarray*}
  \pi_1(N,v,H)&=&\sum\limits_{l\in U(H)(1)}Sh_l(U(H),w)=\frac{1}{24}+\frac{1}{24}+\frac{1}{24}=\frac{1}{8}=\pi_2(N,v,H)=\pi_3(N,v,H),\\
  \pi_4(N,v,H)&=&\sum\limits_{l\in U(H)(4)}Sh_l(U(H),w)=3\frac{1}{24}+2\frac{1}{24}=\frac{5}{24}=\pi_5(N,v,H)=\pi_6(N,v,H).
\end{eqnarray*}
So, this example shows that the position value can be expressed by the uniform hyperlink game.
\end{example}

By Definition \ref{de1} and the proof of Theorem \ref{thm1}, we note that $\eta(H)$ is the key to guaranteeing Eq. (\ref{formu1})  in the proof of Theorem \ref{thm1}.
Clearly,
if we consider an integral multiple of $\eta(H)$ instead of $\eta(H)$, Eq. (\ref{formu1}) is still true.
Based on this observation,
we can construct the other induced games from the original hypergraph game and describe the position value for the hypergraph game by the induced games.

\begin{word}\label{de2}
For any $(N,v,H)\in HCS^N$ and any positive integer $k\ge 1$, let $\eta(H)$ is the least common multiple of the numbers in $\{|e|\,|\,e\in H\}$ and $\rho(k)=k\cdot \eta(H)$. The {\bf $k$-augmented uniform hyperlink game} $(U(H)^k, w^k)$ is defined as follows.
\begin{eqnarray}
&\,&\ U(H)^k(i,e)=\{(i,e,t)\,|\,t=1,\ldots,  \rho(k)\cdot|e|^{-1}\},\\
&\,&\ U(H)^k(i)=\bigcup_{e\in H_i}U(H)^k(i,e)\, \text{ and }\,  U(H)^k(e)=\bigcup_{i\in e}U(H)^k(i,e), \nonumber\\
&\,&\ U(H)^k=\bigcup_{i\in N}U(H)^k(i)=\bigcup_{e\in H}U(H)^k(e)\\
&\,&\ w^k(S)=\sum\limits_{R\in N/H[S]}v(R), \, \text{ for all } S\subseteq U(H)^k,
\end{eqnarray}
where $H[S]=\{e\in H \,|\, U(H)^k(e)\subseteq S\}$.
\end{word}
 By the definition  of the $k$-augmented uniform hyperlink game,  we can obtain the following strengthening of Theorem \ref{thm1}.
 Its  proof  is similar to that of Theorem \ref{thm1}, so we omit it.
\begin{theorem}\label{thm2}
For any $(N, v, H)\in HCS^N$ and any $i\in N$,
$$\pi_i(N,v,H)=\sum_{l\in U(H)^k(i)}Sh_l(U(H)^k,w^k).$$
\end{theorem}

Clearly, when $k=1$,  the $k$-augmented uniform hyperlink game coincides with the uniform hyperlink game, and Theorem \ref{thm1} is a special case of Theorem \ref{thm2} as well. This result will serve to characterize the position value  axiomatically  for arbitrary hypergraph games in the next subsection.

\subsection{The axiomatic characterization of the position value}

In this subsection we provide an axiomatic characterization of the position value for arbitrary hypergraph communication situations.

Before we introduce the  properties   that fully characterize the position value, we first give the following key lemma.
\begin{lemma}\label{le1}
For any hypergraph game $(N,v,H)\in HCS^N$, any $i\in N$, $i\in e\in H$ and $l'\in U(H)(e)$, we have
$$\pi_i\big(N,v,H\setminus\{e\}\big)=\sum\limits_{l\in (U(H)\setminus\{l'\})(i)}Sh_l\big(U(H)\setminus\{l'\},w_{U(H)\setminus\{l'\}}\big).$$
\end{lemma}
\begin{proof} To show the result, we distinguish two cases depending on whether or not the least common multiples $\eta(H)$ and $\eta(H\setminus\{e\})$
are the same.

{\em Case 1.} $\eta(H\setminus\{e\})=\eta(H)$. Then, by Theorem \ref{thm1}, we have
\begin{equation}\label{eq9}
\pi_i(N,v,H\setminus \{e\})=\sum\limits_{l\in U(H\setminus \{e\})(i)}Sh_l(U(H\setminus \{e\}),w_{U(H\setminus \{e\})}).
\end{equation}
Therefore, it is sufficient to show the following equality.
\begin{equation}\label{eq10}
 \sum\limits_{l\in U(H\setminus \{e\})(i)}Sh_l(U(H\setminus \{e\}),w_{U(H\setminus \{e\})})=\sum\limits_{l\in (U(H)\setminus\{l'\})(i)}Sh_l(U(H)\setminus\{l'\},w_{U(H)\setminus\{l'\}}).
\end{equation}

 By the definition of $U(H)$, we have $U(H\setminus\{e\})=U(H)\setminus U(H)(e)\subseteq U(H)\setminus\{l'\}$  for any $e\in H$, where $l'\in U(H)(e)$.
 So, for any $K\subseteq U(H)\setminus\{l'\}$,   it follows from the definition of $w$ that
$w_{U(H)\setminus\{l'\}}(K)=w_{U(H\setminus\{e\})}(K\cap U(H\setminus\{e\}))=w(K)$. This means that
 the players in $\overline{U}(H,l')=U(H)\setminus\{l'\}\setminus U(H\setminus\{e\})$ are null players of $w_{U(H)\setminus\{l'\}}$. Hence, it is easy to see that
\begin{equation}\label{eq11}
Sh_l(U(H)\setminus\{l'\},w_{U(H)\setminus\{l'\}})=\left\{\begin{array}{cl}
                    Sh_l(U(H\setminus\{e\}),w_{U(H\setminus\{e\})}) & \text{ if } l\in U(H\setminus\{e\}),\\
                    0 & \text{ if } l\in \overline{U}(H,l').
                  \end{array}
                  \right.
\end{equation}
This implies that Eq. (\ref{eq10}) holds.

{\em Case 2.} $\eta(H\setminus\{e\})\neq\eta(H)$.  Since both $\eta(H)$ and $\eta(H\setminus\{e\})$ are the
least common multiples of the numbers in $\{|e'|\,|\, e'\in H\}$ and $\{|e'|\,|\, e'\in H\setminus\{e\}\}$, respectively, $\eta(H)=|e|\cdot\eta(H\setminus\{e\})$.
Therefore, by Theorem \ref{thm2}, we have
\begin{equation}\label{eq12}
\pi_i(N,v,H\setminus\{e\})=\sum_{l\in U(H\setminus\{e\})^k(i)}Sh_l(U(H\setminus\{e\})^k,w_{U(H\setminus\{e\})^k}^k).
\end{equation}
where $k=|e|$. Thus, it is sufficient to show the following equality.
\begin{equation*}\label{eq13}
 \sum_{l\in U(H\setminus\{e\})^k(i)}Sh_l(U(H\setminus\{e\})^k,w_{U(H\setminus\{e\})^k}^k)=\sum\limits_{l\in (U(H)\setminus\{l'\})(i)}Sh_l(U(H)\setminus\{l'\},w_{U(H)\setminus\{l'\}}).
\end{equation*}
The following proof is similar to  the proof described in Case 1, and is left to the reader.

Summing up the two cases, it completes the proof of Lemma \ref{le1}.
\end{proof}

The position value  can be expressed  by employing the Harsanyi dividends. Firstly, the uniform hyperlink game $(U(H),w)$ associated with a hypergraph communication situations $(N,v,H)$ can be represented by a unique linear combination of unanimity games, i.e.,
\begin{equation}\label{eq4}
w=\sum\limits_{K\subseteq U(H)}\lambda_K(w)u_K.
\end{equation}
 By Theorem \ref{thm1}, the position value for $(N,v,H)$ can  be expressed in terms of the unanimity coefficients, i.e., Harsanyi dividends, of the associated uniform hyperlink game.  Formally, for any $i\in N$, we have
\begin{eqnarray}\label{eq11}
\pi_i(N,v,H)&\ = &\ \sum\limits_{l\in U(H)(i)}Sh_l(U(H),w) \notag \\
&\ = &\ \sum\limits_{l\in U(H)(i)}\sum\limits_{K\subseteq U(H): l\in K}\frac{\lambda_K(w)}{|K|} \notag \\
&\ = &\ \sum\limits_{K\subseteq U(H)}\lambda_K(w)\frac{|K_i|}{|K|}
\end{eqnarray}
where $K_i=K\cap U(H)(i)$ and the second equality follows from the alternative description of the Shapley value.

We now give the following properties for an allocation rule $f$. The first property is a standard property, called {\em component efficiency}, which already was used to characterize the Myerson value for communication situations, including graph games, conference structures and hypergraph games (\cite{my1,my2} and \cite{van}). It was also used to characterize the position value for graph games (\cite{sl}) and cycle-free hypergraph games (\cite{van}).

{\bf Component efficiency}: For any $(N,v,H)\in HCS^N$ and any $T\in N/H$, it holds that
$$\sum\limits_{i\in T}f_i(N,v,H)=v(T).$$

The second property, called {\em partial balanced conference contributions},  is developed from the {\em balanced link contributions}. The balanced link contributions  is used to characterize the position value for graph games in \cite{sl} and can be expressed as follows.

{\bf Balanced link contributions}: For any $(N,v,L)$ and any $i, j\in N$, it holds that
$$\sum\limits_{e\in L_j} [f_i(N,v,L)-f_i(N,v,L\setminus\{e\})]=\sum\limits_{e\in L_i}[f_j(N,v,L)-f_j(N,v,L\setminus\{e\})].$$

A natural extension of the above property to hypergraph games is the ``balanced hyperlink contributions" (or called balanced conference contributions), which can be given by

{\bf Balanced conference contributions}: For any $(N,v,H)$ and any $i, j\in N$, it holds that
$$\sum\limits_{e\in H_j} [f_i(N,v,L)-f_i(N,v,L\setminus\{e\})]=\sum\limits_{e\in H_i}[f_j(N,v,L)-f_j(N,v,L\setminus\{e\})].$$
However, we note that the obvious extension  fails
to characterize the position value axiomatically
for hypergraph games and this point will be shown in Example 2.
For solving the characterization problem of the position value for hypergraph games, we introduce the partial balanced conference contributions.
The partial balanced conference contributions also deals with the gains players contribute to each other.  When a hyperlink related to a player is broken or built, the threat or contribution of another player received is not only depending on the first player, but also depending on those players whom adjacent to the first player according to the broken or built hyperlink.
Formally, the property can  be expressed as follows.

{\bf Partial balanced conference contributions}: For any $(N,v,H)\in HCS^N$ and any $i, j\in N$, it holds
$$\sum\limits_{e\in H_j}\frac{1}{|e|}[f_i(N,v,H)-f_i(N,v,H\setminus\{e\})]=\sum\limits_{e\in H_i}\frac{1}{|e|}[f_j(N,v,H)-f_j(N,v,H\setminus\{e\})].$$

The partial balanced conference contributions states that the contribution or threat from a player towards another player equals the reverse contribution or threat, where the contribution or threat of a player towards another player is the sum of a portion payoff differences a player can inflict on another player by building or breaking one of his hyperlinks.
  In particular, if $H$ is $r$-uniform, then the property coincides with the balanced conference contributions. But, in general, this property  is obviously different from the balanced hyperlink contributions.

We are ready to show that the position value for the hypergraph games satisfies the two properties we mentioned above.
\begin{lemma}\label{le2}
The position value for hypergraph communication situations satisfies component efficiency and partial balanced conference contributions.
\end{lemma}
\begin{proof}
It has been verified that the position value satisfies component efficiency  by \cite{van}.
We next show that the position value satisfies partial balanced conference contributions.
 Let $(N,v,H)\in HCS^N$ and $i,j\in N$ such that $i\neq j$. Then we have
\begin{eqnarray*}
&\,&\  \sum\limits_{e\in H_j}\frac{1}{|e|}[\pi_i(N,v,H)-\pi_i(N,v,H\setminus\{e\})] \\
&\,&\  =\sum\limits_{e\in H_j}\frac{|U(H)(j,e)|}{\eta(H)}[\pi_i(N,v,H)-\pi_i(N,v,H\setminus\{e\})]\\
&\,&\  =\frac{1}{\eta(H)}\sum\limits_{e\in H_j}\sum_{l\in U(H)(j,e)}\Big[\sum\limits_{K\subseteq U(H)}\lambda_K(w)\frac{|K_i|}{|K|}-\sum\limits_{K\subseteq U(H)\setminus\{l\}}\lambda_K(w)\frac{|K_i|}{|K|}\Big]\\
&\,&\ =\frac{1}{\eta(H)}\sum\limits_{l\in U(H)(j)}\sum\limits_{K\subseteq U(H):l\in K}\lambda_K(w)\frac{|K_i|}{|K|}\\
&\,&\ =\frac{1}{\eta(H)}\sum\limits_{K\subseteq U(H)}|K_j|\lambda_K(w)\frac{|K_i|}{|K|}\\
&\,&\ = \frac{1}{\eta(H)}\sum\limits_{K\subseteq U(H)}\lambda_K(w)\frac{|K_i|\cdot|K_j|}{|K|}\\
&\,&\ = \sum\limits_{e\in H_i}\frac{1}{|e|}[\pi_j(N,v,H)-\pi_j(N,v,H\setminus\{e\})],
\end{eqnarray*}
where the first equality follows from the definition of $U(H)(j,e)$ and $l\in U(H)(j,e)$, the second equality follows  from Eq. (\ref{eq11}) and Lemma \ref{le1} (note that $\lambda_K(w_{U(H)\setminus\{l\}})=\lambda_K(w)$ for any $K\subseteq {U(H)\setminus\{l\}}\subseteq U(H)$), and the last equality follows from the symmetry of player $i$ and $j$.
\end{proof}
The following example illustrates that the position value satisfies the partial balanced conference contributions.
\begin{example}
Consider the hypergraph game as described in Example \ref{ex1}.  The payoffs for several (sub-)hypergraphs according to the position value are shown as follows.
\begin{eqnarray*}
\pi(N,v,A)=\left\{
\begin{array}{cc}
  (\frac{1}{8},\frac{1}{8},\frac{1}{8},\frac{5}{24},\frac{5}{24},\frac{5}{24}) & A=H, \\
  (0,0,0,0,0,0) & A\subsetneq H.
\end{array}
\right.
\end{eqnarray*}
By the definition of the partial balanced conference contributions, the total contribution of player 6 to player 1 equals $\frac{1}{2}(\frac{1}{8}-0)+\frac{1}{3}(\frac{1}{8}-0)=\frac{5}{48}$, by breaking the hyperlink $e_3$ and $e_4$, respectively. The reverse contribution of player 1 to player 6 equals $\frac{1}{2}(\frac{5}{24}-0)=\frac{5}{48}$ as well. Hence, the position value satisfies the partial balanced conference contributions. However, this example also shows that the position value does not satisfies the balanced conference contributions. To be specific, the contribution of player 6 to player 1 equals $(\frac{1}{8}-0)+(\frac{1}{8}-0)=\frac{2}{4}$, while the reverse contribution of player 1 to player 6 equals $\frac{5}{24}-0=\frac{5}{24}$.
\end{example}

We now can provide a characterization of the position value by  Lemma \ref{le2}. Its proof is similar to the proof of Theorem 3.1 in \cite{sl} which characterizes the position value for arbitrary (graph) communication situations.
\begin{theorem}\label{th2}
The position value for hypergraph communication situations is the unique allocation rule that satisfies component efficiency and partial balanced conference contributions.
\end{theorem}
\begin{proof} By Lemma \ref{le2},
it is proved that the position value for hypergraph games satisfies component efficiency and partial balanced conference contributions.
 It remains
to show  that  the position value is the unique value that satisfies the two properties.  Suppose $f$ is an allocation rule satisfies the two properties, we show that
$f=\pi$.
We proceed by induction on  $|H|$.
For $|H|=0$, the assertion  immediately follows from component efficiency. Next we may assume that $f$ coincides with the position value $\pi$ if $|H|\leq k-1$. We consider the case when $|H|=k$.
For any component $C\in N/H$, let  $C=\{1,2,\ldots, c\}$.  We can obtain the following system of linearly independent equations by the two properties and the hypothesis,
\begin{eqnarray*}
&\,&\ \sum\limits_{e\in H_2}\frac{1}{|e|}f_1(H)-\sum\limits_{e\in H_1}\frac{1}{|e|}f_2(H)=\sum\limits_{e\in H_2}\frac{1}{|e|}\pi_1(H\setminus\{e\})-\sum\limits_{e\in H_1}\frac{1}{|e|}\pi_2(H\setminus\{e\}),\\
&\,&\\
&\,&\ \cdots \\
&\,&\ \sum\limits_{e\in H_c}\frac{1}{|e|}f_1(H)-\sum\limits_{e\in H_1}\frac{1}{|e|}f_c(H)=\sum\limits_{e\in H_c}\frac{1}{|e|}\pi_1(N,v,H\setminus\{e\})-\sum\limits_{e\in H_1}\frac{1}{|e|}\pi_c(H\setminus\{e\}),\\
&\,&\\
&\,&\ \sum\limits_{i\in T}f(N,v,H)=v(T),
\end{eqnarray*}
where write $f_i(H)$ and $f_i(H\setminus\{e\})$ instead of $f_i(N,v,H)$ and  $f_i(N,v,H\setminus\{e\})$, respectively, for $i=1,2,\ldots, c$. One may easily verify that the above system has a unique solution. Since the position value satisfies component efficiency and partial balanced conference contributions,  the position value is  a solution of the above system.
Consequently, we conclude that $f=\pi$ for any hypergraph communication situations with $|H|=k$.
\end{proof}

\section{Concluding remarks}
In this paper we provides the non-axiomatic  characterization and axiomatic characterization of the position value for arbitrary hypergraph communication situations.
 Here the non-axiomatic  characterization is in line with the works of \cite{ca1} and \cite{ko}, in which non-characterizations of the position are provided by the Myerson value and the Shapley value of modifications of communication situations, respectively.  \cite{ca1} expressed the position value for hypergraph communication situations in terms of the Myerson value by applying the hyperlink agent form (HAF).
 We now give a comparison between the expressions of the position value in \cite{ca1} and in our paper.

We first recall the definition of the hyperlink agent from $HAF(N,v,H)=(\bar{N},\bar{v},\bar{H})$, where the player set is \\
\begin{equation*}
\bar{N}=\bigcup_{i\in N} \bar{N}(i), \quad \bar{N}(i)=\bigcup_{h\in H_i}\bar{N}(i,h), \quad \bar{N}(i,h)=\{(i,h,k)\,|\, k=1,2,\ldots,  \eta(H)\cdot|h|^{-1}\},
\end{equation*}
and the set of hyperlinks is
\begin{equation*}
\bar{H}=\bar{H}^o\cup\bigcup_{i\in N}L^{\bar{N}(i)}, \quad \bar{H}^o=\{\bar{h}\,|\, h\in H\}, \quad \bar{h}=\bigcup_{i\in h}\bar{N}(i,h),
\end{equation*}
and the characteristic function is
\begin{equation*}
\bar{v}(\bar{K})=v(N(\bar{K})), \quad N(\bar{K})=\{i\in N\,|\, \bar{N}(i)\cap\bar{K}\neq \emptyset\}.
\end{equation*}
According to the definitions of the HAF and the uniform hyperlink game, it is easy to check that the differences between the two induced games lies in two aspects: the structures and the characteristic functions. However, somewhat surprisingly, we have the following relationships between them.
\begin{corollary}
For any $(N,v,H)\in HCS^N$, and the uniform hyperlink game $(U(H),w)$, hyperlink agent form (HAF) $(\bar{N},\bar{v},\bar{H})$ defined on it, we have $Sh(U(H),w)=\mu(\bar{N},\bar{v},\bar{H})$.
\end{corollary}
\begin{proof}
By the definition of the player sets $U(H)$ and $\bar{N}$ and the Myerson value, it is sufficient to show that $w(S)=\bar{v}^{\bar{H}}(S)$ for any $S\subseteq U(H)=\bar{N}$.

In fact, for any $S\subseteq \bar{N}$, we have
\begin{equation*}
\bar{v}^{\bar{H}}(S)=\sum_{C\in S/\bar{H}}\bar{v}(C)=\sum_{C\in S/\bar{H}}v(N(C))=\sum_{C\in N(S)/H[S]}v(C)=\sum_{C\in N/H[S]}v(C)=w(S),
\end{equation*}
where the third equation follows from the definition of $\bar{H}$ and $H[S]$, and the fourth equation holds by the zero-normalized game.
\end{proof}
Even though this corollary shows that the Shapley payoff of the uniform hyperlink game and the Myerson payoff of the HAF coincide with each other, the uniform hyperlink game seems
to be more concise than HAF, more importantly, the uniform hyperlink game can provide a powerful assistance in characterizing the position value for arbitrary hyperlink communication situations.


\small\begin{thebibliography}{99}

\bibitem[Algaba et al.(2000)]{al}Algaba E, Bilbao JM, Borm P, L\'opez JJ (2000) The position value for union stable systems. Math Methods
Oper Res 52:221--236



\bibitem[B\'eal et al.(2012)]{brs} B\'eal S, R\'emila E, Solal P (2012) Fairness and fairness for neighbors: The difference between the
Myerson value and component-wise egalitarian solutions. Economics Letters 117(1):263--267


\bibitem[Born et al.(1992)]{Bo} Borm P, Owen G, Tijs S (1992) On the position value for communication situations. SIAM J Discr Math
5:305--320

\bibitem[Casajus(2007)]{ca1} Casajus A (2007) The position value is the Myerson value, in a sense. Int J Game Theory 36:47--55



\bibitem[Harsanyi(1959)]{har} Harsanyi JC (1959) A bargaining model for cooperative $n$-person games. In: Tucker AW, Luce RD (eds) Contributions to the theory of games IV. Princeton University Press, Princeton, pp 325--355

\bibitem[Herings et al.(2008)]{he}Herings PJJ, van der Laan G ,Talman AJJ (2008) The average tree solution for cycle-free graph games. Games Econ Behav 62(1): 77--92




\bibitem[Kongo(2010)]{ko} Kongo T (2010) Difference between the position value and the Myerson
value is due to the existence of coalition structures. Int J Game Theory (2010) 39:669--675

\bibitem[Meessen(1988)]{Me}Meessen R (1988) Communication games, Master¡¯s thesis, Department of Mathematics. University of
Nijmegen, the Netherlands (in Dutch)

\bibitem[Myerson(1977)]{my1}Myerson RB (1977) Graphs and cooperation in games. Math Oper Res 2:225--229

\bibitem[Myerson(1980)]{my2} Myerson RB (1980) Conference structures and fair allocation rules. Int J Game Theory 9:169--182





\bibitem[Shan et al.(2016)]{szd} Shan E, Zhang G, Dong Y (2016) Component-wise proportional solutions for communication graph games. Mathematical Social Sciences 81:22--28

\bibitem[Shapley(1953)]{sha} Shapley LS (1953) A value for $n$-person games. In: Kuhn H, Tucker AW (eds) Contributions to the Theory
of Games II. Princeton University Press, Princeton, pp 307--317

\bibitem[Slikker(2005)]{sl} Slikker M (2005) A characterization of the position value. Int J Game Theory 33:505--514


\bibitem[van den Brink et al.(2012)]{bkl} van den Brink R, Khmelnitskaya A, van der Laan G (2012) An efficient and fair solution for communication graph games. Economics Letter 117(3):786--789

\bibitem[van den Brink et al.(2011)]{blp} van den Brink R, van der Laan G, Pruzhansky V (2011) Harsanyi power solutions for graph-restricted games. Int J Game Theory 40:87--110


\bibitem[van den Nouweland et al.(1992)]{van} van den Nouweland A, Borm P, Tijs S (1992) Allocation rules for hypergraph communication situations.
Int J Game Theory 20:255--268


\end{thebibliography}
\end{document}